\documentclass[11pt]{amsart}
\usepackage{amssymb}
\usepackage{graphicx}
\usepackage{amscd}

\newtheorem{theorem}{Theorem}[section]
\theoremstyle{plain}
\newtheorem{acknowledgement}{Acknowledgement}

\newtheorem{corollary}{Corollary}[section]

\newtheorem{definition}{Definition}[section]

\newtheorem{lemma}{Lemma}[section]

\newtheorem{proposition}{Proposition}[section]

\numberwithin{equation}{section}

\input{tcilatex}

\begin{document}
\title[Local controllability]{A necessary and sufficient condition for local
controllability around closed orbits.}
\author{Marek Grochowski}
\email{m.grochowski@uksw.edu.pl}

\begin{abstract}
In this paper we give a necessary and sufficient condition for local
controllability around closed orbits for general smooth control systems. We
also prove that any such system on a compact manifold has a closed orbit.
\end{abstract}

\maketitle

\section{Introduction.}

\subsection{Motivation.}

The aim of this note is to formulate and prove a necessary and sufficient
condition for local controllability of general control systems around a
closed orbit. Let $M$ be a smooth (or real analytic) manifold, and let $%
\mathcal{U}$ be a subset of $\mathbb{R}^{k}$. Consider a smooth (or real
analytic) control system $(\Sigma )$ $\dot{x}=f(x,u)$, $u\in \mathcal{U}$,
where controls $u:[0,T]\longrightarrow \mathcal{U}$ are bounded measurable,
and the final time $T=T(u)\geq 0$ is not fixed and depends on a control $u$.
If $u:[0,T(u)]\longrightarrow \mathcal{U}$ is a control then a solution of
the ordinary differential equation $\dot{x}(t)=f(x(t),u(t))$ is called 
\textit{a trajectory (or an admissible curve, or an orbit) of }$(\Sigma )$%
\textit{\ generated by }$u$. The system $(\Sigma )$ is said to be \textit{%
controllable} if for every $x,y\in M$ there exists a control $u$ defined on $%
[0,T(u)]$ such that if $\gamma $ is the trajectory of $(\Sigma )$ generated
by $u$ and satisfying $\gamma (0)=x$, then $\gamma (T(u))=y$. The system $%
(\Sigma )$ is \textit{controllable at a point} $x$ if there exists a
neighbourhood $U$ of $x$ such that the restriction of $(\Sigma )$ to $U$ is
a controllable system. A neighbourhood\ $U$ as above is called \textit{a
controllable neighbourhood}. There are a lot of results devoted to
controllability question for control systems in connection with the
existence of closed or 'almost closed' orbits, for instance: \cite{Bonnard1}%
, \cite{Bonnard2}, \cite{Hermes1}, \cite{Hermes2}, \cite{Lobry1}, \cite%
{Lobry2}. Before we cite a few of them, we will fix some notation. If $%
Z_{1},...,Z_{l}$ are vector fields on a manifold $M$ then denote by $%
Lie\{Z_{1},...,Z_{l}\}$ the Lie algebra generated by $Z_{1},...,Z_{l}$. For
an $x\in M$, let $Lie_{x}\{Z_{1},...,Z_{l}\}$ stand for the subspace in $%
T_{x}M$ spanned by all vectors $v$ of the form $v=W(x)$ where $W\in
Lie\{Z_{1},...,Z_{l}\}$. Recall that a point $x$ is Poisson stable for a
vector field $X$ if for every neighbourhood $V$ of $x$, and for every $T>0$
there exist $t_{1},t_{2}>T$ such that $g_{X}^{t_{1}}(x)\in V$, $%
g_{X}^{-t_{2}}(x)\in V$. Also, a vector field $X$ defined on a Riemannian
manifold is conservative if $g_{X}^{t}$ preserves the natural measure on $M$%
. In both cases $g_{X}^{t}$ stands for the flow of $X$.

Let us start from citing two results on global controllability. \smallskip

\textit{Theorem (Bonnard \cite{Bonnard1}): Consider an affine control system 
}$\dot{x}=X+\tsum_{i=1}^{k}u_{i}Y_{i}$\textit{\ on an analytic manifold }$M$%
\textit{, where }$\tsum_{i=1}^{k}\left\vert u_{i}\right\vert \leq 1$\textit{%
\ and the fields }$X,Y_{i}$, $i=1,...,k$, \textit{are supposed to be
analytic. Assume that the set of points which are Poisson stable for }$X$%
\textit{\ is dense in }$M$\textit{. Then the system in question is
controllable if and only if }$\dim Lie_{x}\{X,Y_{1},...,Y_{k}\}=\dim M$%
\textit{\ for every }$x\in M$\textit{. }\smallskip

In particular, controllability holds if all orbits of $X$ are
closed.\smallskip

\textit{Theorem (Lobry \cite{Lobry2}): Consider an affine control system }$%
\dot{x}=X+\tsum_{i=1}^{k}u_{i}Y_{i}$\textit{\ on a compact analytic manifold 
}$M$\textit{, where }$\tsum_{i=1}^{k}\left\vert u_{i}\right\vert \leq 1$ and 
\textit{the fields }$X,Y_{i}$, $i=1,...,k$, \textit{are supposed to be
analytic and conservative. Then the system in question is controllable if
and only if }$\dim Lie_{x}\{X,Y_{1},...,Y_{k}\}=\dim M$\textit{\ for every }$%
x\in M$\textit{. }\smallskip

Two last theorems are not exact quotations but can be deduced respectively
from \textit{\cite{Bonnard1} and \cite{Lobry2}}.

There are also results concerning local controllability. The result which is
closest to our interests is as follows.\smallskip

\textit{Theorem (Nam, Arapostathis \cite{Arapo}): Consider a smooth control
system }$\dot{x}=X+\tsum_{i=1}^{k}u_{i}Y_{i}$\textit{, }$u\in \mathcal{U}$,%
\textit{\ where }$\mathcal{U}$\textit{\ is a neighbourhood of }$0$\textit{,
and let }$\Gamma $\textit{\ be a closed orbit for }$X$\textit{. Define }$%
\mathcal{G}_{i}=\{ad^{i}X.Y_{j}:j=1,...,k\}$\textit{, and suppose that there
exists a point }$x\in \Gamma $\textit{\ such that }%
\begin{equation}
rank\{X,\mathcal{G}_{0},\mathcal{G}_{1},...\}(x)=\dim M\text{.}
\label{Arwar}
\end{equation}%
\textit{Then }$\Gamma $\textit{\ has a controllable neighbourhood.\smallskip 
}

There are also other results, cf. for instance \cite{Hermes1}, but they use
stronger assumptions than \cite{Arapo}. As it will be seen at the end of
this paper, assumptions in \cite{Arapo} can be weakened.

\subsection{Statement of main results.}

The goal of this paper (which generalizes some ideas from the sub-Lorentzian
geometry that were developed by the author in \cite{gr2}) is to prove two
theorems: one concerns the existence of closed orbits, the other states
necessary and sufficient conditions for local controllability around closed
orbits. In order to state them, we first formulate our assumptions. Again,
let 
\begin{equation}
\dot{x}=f(x,u)=f_{u}(x)\text{, \ \ }u\in \mathcal{U}\text{,}  \tag{$\Sigma $}
\end{equation}%
be a control system, where $M$ is a smooth manifold, $\mathcal{U}$ is (an
arbitrary) subset of $\mathbb{R}^{k}$, $f$ is a continuous mapping $M\times 
\mathcal{U}\longrightarrow TM$, and $f_{u}$ is a smooth vector field on $M$
for every $u\in \mathcal{U}$. Our main assumption is 
\begin{equation}
\dim Lie_{x}\{f_{u}:\;u\in \mathcal{U\}}=n=\dim M  \label{Ass2}
\end{equation}%
for every $x\in M$. Similarly as above, our controls are bounded measurable
and the final time is not fixed. It follows from known results for ODE's
with measurable right hand side (see e.g. \cite{Bress}) that under such
assumptions, to every control $u:[0,T]\longrightarrow \mathcal{U}$ there
corresponds an admissible trajectory of $(\Sigma )$ (defined maybe on a
smaller interval).

The first result that we will prove is the following

\begin{theorem}
Consider the control system $(\Sigma )$ for which (\ref{Ass2}) holds, and
suppose that $M$ is compact. Then the system $(\Sigma )$ has closed orbits.
\end{theorem}

Let $x\in M$ and take its neighbourhood $U$. Denote by $\mathcal{A}^{+}(x,U)$
\textit{the reachable set from a point }$x$\textit{\ in }$U$ for the system $%
(\Sigma )$, i.e. the set of endpoints of all trajectories of $(\Sigma )$
that start from $x$, are generated by measurable controls (final time is not
fixed), and are contained in $U$. The sets $\mathcal{A}^{+}(x,M)$ will be
denoted simply by $\mathcal{A}^{+}(x)$. Let us remark that controllability
of $(\Sigma )$ means that $\mathcal{A}^{+}(x)=M$ for every $x\in M$.

Suppose now that $\Gamma $ is a closed orbit for $(\Sigma )$. If a point $x$
belongs to $\Gamma $ then $\Gamma _{x}$ will stand for the set $\Gamma
\backslash \{x\}$.

\begin{definition}
We say that the closed orbit $\Gamma $ is \textit{regular}, if there exists
a point $x\in \Gamma $ and a neighbourhood $U$ of $x$ such that 
\begin{equation}
\Gamma _{x}\cap \mathcal{A}^{+}(x,U)\subset int\text{ }\mathcal{A}^{+}(x,U)%
\text{.}  \label{Reg}
\end{equation}
\end{definition}

Our second result can be stated as follows.

\begin{theorem}
Suppose that $\Gamma $ is a closed orbit for the system $(\Sigma )$ for
which (\ref{Ass2}) holds. Then the necessary and sufficient condition for $%
(\Sigma )$ to be locally controllable at every point of $\Gamma $ is that $%
\Gamma $ be a regular closed orbit. More precisely, a closed orbit $\Gamma $
of $(\Sigma )$ is regular if and only if $\Gamma $ possesses a controllable
neighbourhood.
\end{theorem}

Note that in theorem 1.2 $M$ is not supposed to be compact. Let us also note
that the curve $\Gamma $ need not be smooth. Theorem 1.2 generalizes
slightly results from \cite{Arapo} as it will be clarified at the end of
this paper.

\section{Proofs of Theorems.}

Along with the system $(\Sigma )$ we will consider the system 
\begin{equation}
\dot{x}=-f(x,u)\text{, \ \ }u\in \mathcal{U}\text{.}  \tag{$\Sigma ^{-}$}
\end{equation}

Let us note a simple observation which will be useful later.

\begin{lemma}
$\gamma (t)$ is a trajectory of the system $(\Sigma )$ generated by a
control $u(t)$ if and only if $\tilde{\gamma}(t)=\gamma (-t)$ is a
trajectory of the system $(\Sigma ^{-})$ generated by a control $\tilde{u}%
(t)=u(-t)$.
\end{lemma}

Denote by $\mathcal{A}^{-}(x,U)$ the corresponding reachable set from $x$
for the system $(\Sigma ^{-})$. At the same time let $\mathcal{A}_{0}^{+}(x)$%
, $\mathcal{A}_{0}^{-}(x)$ be the reachable sets for $(\Sigma )$ and $%
(\Sigma ^{-})$, respectively, generated by piecewise constant controls.
Recall now \cite{Krener} Krener's theorem which states that under the
assumption (\ref{Ass2}) the inclusion $\mathcal{A}_{0}^{+}(x)\subset 
\overline{int\text{ }\mathcal{A}^{+}(x)}$ (and the same for $\mathcal{A}%
_{0}^{-}(x)$) holds true. Therefore $int$ $\mathcal{A}^{+}(x)$ and $int$ $%
\mathcal{A}^{-}(x)$ are non-empty for every $x\in M$. Notice also that 
\begin{equation*}
x\in \overline{int\text{ }\mathcal{A}^{+}(x)}\cap \overline{int\text{ }%
\mathcal{A}^{-}(x)}
\end{equation*}%
for any $x\in M$. Indeed, by Krener's theorem 
\begin{equation*}
x\in \mathcal{A}_{0}^{+}(x)\subset \overline{int\text{ }\mathcal{A}%
_{0}^{+}(x)}\subset \overline{int\text{ }\mathcal{A}^{+}(x)}\text{,}
\end{equation*}%
and the same for $\mathcal{A}^{-}(x)$. Now it is easy to show that

\begin{lemma}
\label{rez1}$y\in int$ $\mathcal{A}^{+}(x)$ if and only if $x\in int$ $%
\mathcal{A}^{-}(y)$.
\end{lemma}

\begin{proof}
Suppose that $y\in int$ $\mathcal{A}^{+}(x)$. Since $y\in \overline{int\text{
}\mathcal{A}^{-}(y)}$ it follows that $\allowbreak int$ $\mathcal{A}%
^{+}(x)\cap \allowbreak int$ $\mathcal{A}^{-}(y)\neq \varnothing $. Taking a 
$z\in int$ $\mathcal{A}^{+}(x)\cap int$ $\mathcal{A}^{-}(y)$ we see that
there exist admissible curves for the system ($\Sigma )$: $\sigma _{1}$
joining $x$ to $z$, and (cf. lemma 2.1) $\sigma _{2}$ joining $z$ to $y$.
Reversing time in $\sigma _{1}\cup \sigma _{2}$ we obtain an admissible
curve $\tilde{\sigma}$ for the system $(\Sigma ^{-})$ that joins $y$ to $x$,
and which belongs to the interior $int$ $\mathcal{A}^{-}(y)$ starting from a
certain time $t_{0}>0$ (for instance $t_{0}$ corresponds to a point $z$).
But this means that $\tilde{\sigma}$ stays in $int$ $\mathcal{A}^{-}(y)$ for
all $t>t_{0}$, therefore $x\in int$ $\mathcal{A}^{-}(y)$.
\end{proof}

We come to the proof of theorem 1.1 now. First we need to establish the
following proposition.

\begin{proposition}
The family $\left\{ int\text{ }\mathcal{A}^{+}(x)\right\} _{x\in M}$ forms
an open covering of $M$.
\end{proposition}

\begin{proof}
Fix a point $x\in M$. Send through it a trajectory $\gamma $, $\gamma (0)=x$%
, of $(\Sigma ^{-})$ such that $\gamma (t)\in int$ $\mathcal{A}^{-}(x)$ for
a $t>0$; by our assumptions such a curve exists. Now, the above lemmas imply
that $x\in int$ $\mathcal{A}^{+}(\gamma (t))$, proving the assertion.
\end{proof}

Suppose that $M$ is compact. By proposition 2.1 there are points $%
x_{1},...,x_{m}\in M$ such that $M=\tbigcup_{i=1}^{m}int$ $\mathcal{A}%
^{+}(x_{i})$. Now $x_{1}\in int$ $\mathcal{A}^{+}(x_{i_{1}})$, for an$\
i_{1}\in \{1,...,m\}$, $x_{i_{1}}\in int$ $\mathcal{A}^{+}(x_{i_{2}})$ for $%
i_{2}\in \{1,...,m\}$ etc. In this way we are led to an infinite sequence $%
\{x_{i_{k}}\}_{k=1,2,...}$ with $x_{i_{k}}\in int$ $\mathcal{A}%
^{+}(x_{i_{k+1}})$ and $i_{k}\in \{1,...,m\}$. Therefore we can find
positive integers $l$ and $p$ such that $x_{i_{l}}\in int$ $\mathcal{A}%
^{+}(x_{i_{l+1}})$, $x_{i_{l+1}}\in int$ $\mathcal{A}^{+}(x_{i_{l+2}})$..., $%
x_{i_{l+p}}\in int$ $\mathcal{A}^{+}(x_{i_{l}})$. This ends the proof of
theorem 1.1.\smallskip

Now we move on to the proof of theorem 1.2. First of all let us list
immediate properties of closed orbits. If $\Gamma $ is a closed orbit for $%
(\Sigma )$ then $\mathcal{A}^{+}(x_{1})\mathcal{=A}^{+}\mathcal{(}x_{2})$
for every $x_{1},x_{2}\in \Gamma $. Moreover, $\mathcal{A}^{+}(x)=\mathcal{A}%
^{+}(\Gamma )$ for $x\in \Gamma $, where by $\mathcal{A}^{+}(\Gamma )$ we
mean $\bigcup_{x\in \Gamma }\mathcal{A}^{+}(x)$. Since $\Gamma $, under
suitable parameterization, is a closed orbit also for $(\Sigma ^{-})$, we
have $\mathcal{A}^{-}(x_{1})=\mathcal{A}^{-}(x_{2})=\mathcal{A}^{-}(\Gamma )$
for any $x_{1},x_{2}\in \Gamma $. Let us also recall a standard fact from
control theory asserting that the reachable set $\mathcal{A}^{\pm }(x)$ is
open if and only if $x\in int$ $\mathcal{A}^{\pm }(x)$.

Next we prove

\begin{lemma}
\label{rez2}If $\Gamma $ is a regular closed orbit for $(\Sigma )$ then the
set $\mathcal{A}^{+}(\Gamma )$ is open.
\end{lemma}

\begin{proof}
Take an $x\in \Gamma $ and $U$ such that (\ref{Reg}) is satisfied, i.e. $%
\Gamma _{x}\cap \mathcal{A}^{+}(x,U)\subset int$ $\mathcal{A}^{+}(x,U)$.
Clearly $int$ $\mathcal{A}^{+}(x,U)\subset int$ $\mathcal{A}^{+}(x)$. Take a
point $y\in \Gamma _{x}\cap \mathcal{A}^{+}(x,U)$ and an open set $V$ such
that $y\in V\subset \mathcal{A}^{+}(x)$. For any $z\in V$ one can construct
a trajectory of $(\Sigma )$ joining $y$ to $z$: we connect $y$ to $x$ by a
suitable segment of $\Gamma $, and then $x$ to $z$ ($z\in \mathcal{A}^{+}(x)$%
). In this way we proved that $V\subset \mathcal{A}^{+}(y)$, i.e. $y\in int$ 
$\mathcal{A}^{+}(y)$. This proves that $\mathcal{A}^{+}(y)=\mathcal{A}%
^{+}(\Gamma )$ is open.
\end{proof}

The last stage in proving theorem 1.2 is the following observation.

\begin{lemma}
Let $\Gamma $ be a closed orbit for $(\Sigma )$. $\Gamma $ is regular for $%
(\Sigma )$ if and only if it is regular for $(\Sigma ^{-})$ (under suitable
parameterization).
\end{lemma}

\begin{proof}
Because of symmetry, it is enough to prove one implication. Suppose that $%
\Gamma $ is regular for $(\Sigma )$ and choose $x_{1}$ and $U$ such that $%
\Gamma _{x_{1}}\cap \mathcal{A}^{+}(x_{1},U)\subset int$ $\mathcal{A}%
^{+}(x_{1},U)\subset int$ $\mathcal{A}^{+}(x_{1})$. Take a point $x_{2}\in
\Gamma _{x_{1}}\cap \mathcal{A}^{+}(x_{1},U)$ and denote by $[x_{1},x_{2}]$
the segment of $\Gamma $ bounded by points $x_{1}$ and $x_{2}$. By lemma \ref%
{rez2}, for every $z\in \lbrack x_{1},x_{2}]$, $x_{2}\in int$ $\mathcal{A}%
^{+}(z)$ which, by lemma \ref{rez1}, means that $z\in int$ $\mathcal{A}%
^{-}(x_{2})$. Thus $[x_{1},x_{2}]\subset int$ $\mathcal{A}^{-}(x_{2})$, and
consequently $\Gamma _{x_{2}}\cap \mathcal{A}^{-}(x_{2},W)\subset int$ $%
\mathcal{A}^{-}(x_{2},W)$ for suitably chosen neighbourhood $W$ of $x_{2}$,
proving that $\Gamma $ is regular for $(\Sigma ^{-})$.
\end{proof}

\begin{corollary}
If $\Gamma $ is a regular closed orbit for $(\Sigma )$ then the set $%
\mathcal{A}^{-}(\Gamma )$ is open.
\end{corollary}

In order to finish the proof of theorem 1.2 it is enough to notice that if $%
\Gamma $ is a regular orbit for $(\Sigma )$ then $U=\mathcal{A}^{+}(\Gamma
)\cap \mathcal{A}^{-}(\Gamma )$ is a controllable neighbourhood. Indeed,
take arbitrary $x,y\in U$. Since $x\in \mathcal{A}^{-}(\Gamma )$ there
exists a trajectory of $(\Sigma )$ joining $x$ to a point of $\Gamma $.
Similarly, since $y\in \mathcal{A}^{+}(\Gamma )$ there exists a trajectory
of $(\Sigma )$ joining a point of $\Gamma $ to $y$. Finally, it is clear
that any two points belonging to $\Gamma $ can be joined by a trajectory of $%
(\Sigma )$. Evidently, any admissible trajectory joining $x$ to $y$ obtained
in this way does not leave $U$ by the very definition of $U$.

\section{One example.}

Before we state our example let us recall a concept of geometric optimality
and so-called singular extremals for the system $(\Sigma )$. So fix a
trajectory $\gamma :[0,T]\longrightarrow U$ of $(\Sigma )$, $U$ being an
open subset of $M$, which is generated by a control $\tilde{u}%
:[0,T]\longrightarrow \mathcal{U}$. We say that $\gamma $ (or $\tilde{u}$)
is \textit{geometrically optimal in }$U$ if $\gamma (T)\in \partial _{U}%
\mathcal{A}^{+}(\gamma (0),U)$; $\partial _{U}$ denotes here the boundary
operator with respect to $U$. On the other hand, $\gamma
:[0,T]\longrightarrow M$ is called an \textit{extremal}, if there exists an
absolutely continuous $p:$ $[0,T]\longrightarrow T^{\ast }M$ (called \textit{%
an extremal lift}) such that $p(t)\in T_{\gamma (t)}^{\ast }M\backslash
\{0\} $ for every $t$, and such that if we set $\mathcal{H}%
_{u}(x,p)=\left\langle p,f_{u}(x)\right\rangle $, then\smallskip

\begin{enumerate}
\item[(i)] $(\dot{\gamma}(t),\dot{p}(t))=\overrightarrow{\mathcal{H}_{\tilde{%
u}(t)}}(\gamma (t),p(t))$ a.e. on $[0,T]$ ($\overrightarrow{\mathcal{H}_{u}}$
is the Hamiltonian vector field on $T^{\ast }M$ corresponding to the
function $(x,p)\longrightarrow \mathcal{H}_{u}(x,p)$),\smallskip

\item[(ii)] $\mathcal{H}_{\tilde{u}(t)}(\gamma (t),p(t))=0$ on $[0,T]$,
and\smallskip

\item[(iii)] $\mathcal{H}_{\tilde{u}(t)}(\gamma (t),p(t))=\max_{u\in 
\mathcal{U}}\mathcal{H}_{u}(\gamma (t),p(t))$ a.e. on $[0,T]$.\smallskip
\end{enumerate}

It is proved \cite{Agr} that a necessary condition for $\gamma $ to be
geometrically optimal is that $\gamma $ be an extremal. Now, an extremal $%
\gamma (t)$ generated by a control $\tilde{u}$ with values in $int$ $%
\mathcal{U}$ is called \textit{a singular extremal} if there exists an
extremal lift $p(t)$ such that additionally\smallskip

\begin{enumerate}
\item[(iv)] $\frac{\partial \mathcal{H}_{u}(\gamma (t),p(t))}{\partial u}%
|_{u=\tilde{u}(t)}=0$ for every $t$.\smallskip
\end{enumerate}

It is a standard fact that if $\gamma $ is a geometrically optimal
trajectory of $(\Sigma )$ generated by a steering $u:[0,T]\longrightarrow
int $ $\mathcal{U}$ with values in $int$ $\mathcal{U}$, then $\gamma $ is a
singular trajectory of $(\Sigma )$.

Consider now a control affine system 
\begin{equation}
\dot{x}=X+uY\text{, \ \ }\left\vert u\right\vert \leq 1\text{,}
\label{AffEx}
\end{equation}%
defined on a manifold $M$. Fix a point $x_{0}$ and a time interval $[0,T]$.
Let $\gamma $ be the trajectory of $X$ initiating at a point $x_{0}$; in
other words $\gamma $ is a trajectory of our control system generated by the
control $u^{0}(t)\equiv 0$. Next, consider the so-called endpoint map $\Phi
^{T,x_{0}}$, i.e. the mapping which to each control $u:[0,T]\longrightarrow
\lbrack -1,1]$ assigns the point $\Phi ^{T,x_{0}}(u)=\gamma _{u}(T)$, where $%
\gamma _{u}$ is the trajectory of (\ref{AffEx}) that starts from $x_{0}$ and
is generated by $u$. It can be proved (see e.g. \cite{Bonnard2}) that 
\begin{equation*}
im\text{ }d_{u^{0}}\Phi ^{T,x_{0}}=Span\{Y(\gamma (T)),\left(
ad^{k}X.Y\right) (\gamma (T)):\text{ }k=1,2,...\}\text{,}
\end{equation*}%
where $adX.Y=\left[ X,Y\right] $, and $ad^{k+1}X.Y=\left[ X,ad^{k}X.Y\right] 
$, $k=1,2,...$ It is known (see again e.g. \cite{Bonnard2}) that $\gamma $
is not a singular trajectory for (\ref{AffEx}) if and only if 
\begin{equation}
\dim Span\{Y(\gamma (T)),\left( ad^{k}X.Y\right) (\gamma
(T)):k=1,2,...\}=\dim M\text{.}  \label{ConSing}
\end{equation}%
Now let us take a closer look at the result from \cite{Arapo} cited in the
introduction, applied to the system (\ref{AffEx}). Suppose that $\Gamma $ is
a closed orbit of $X$ and fix an $x\in \Gamma $. If (\ref{Arwar}) is
satisfied at $x$ then (\ref{ConSing}) does not have to be satisfied, as it
is explained in \cite{Arapo}. On the other hand assume that (\ref{ConSing})
is satisfied at $x$. Then of course (\ref{Arwar}) is also satisfied and, by
the above remark, $\Gamma $ is not a singular trajectory. Consequently, it
is not geometrically optimal from $x$ and, what follows, it is a regular
closed orbit for (\ref{AffEx}). Thus the satisfaction of (\ref{ConSing})
implies that $\Gamma $ is a regular closed orbit.

Now, we are going to present a simple construction of a closed trajectory $%
\Gamma $ which does not satisfy neither (\ref{ConSing}) nor (\ref{Arwar})
but anyway is a regular closed orbits.

To this end consider $W=\left\{ (x_{1},x_{2},x_{3}):\;x_{2}^{2}+x_{3}^{2}<1%
\text{, }0\leq x_{1}\leq 2\pi \right\} \subset \mathbb{R}^{3}$. Let us
introduce the following equivalence relation on $W$: $(x_{1},x_{2},x_{3})%
\thicksim (x_{1}^{\prime },x_{2}^{\prime },x_{3}^{\prime })$ if and only if $%
x_{2}=x_{2}^{\prime }$, $x_{3}=x_{3}^{\prime }$, $x_{1}=0$, $x_{1}^{\prime
}=2\pi $ or $x_{1}=2\pi $, $x_{1}^{\prime }=0$. Consider the factorization $%
p:W\longrightarrow M=W/\thicksim $. The space $M$ is a $3$-dimensional
manifold which in an obvious way is embedded in $\mathbb{R}^{3}$. Let $%
\tilde{X}=\frac{\partial }{\partial x_{1}}+x_{2}^{k}\frac{\partial }{%
\partial x_{3}}$, $\tilde{Y}=\frac{\partial }{\partial x_{2}}$, $k\geq 3$,
be vector fields on $\mathbb{R}^{3}$. After factorization they are
transformed to vector fields 
\begin{equation}
X=p_{\ast }\tilde{X}\text{, \ \ }Y=p_{\ast }\tilde{Y}  \label{PolaEx}
\end{equation}%
on $M$. Now denote by $(\Sigma )$ the control system (\ref{AffEx}) on $M$
where $X$ and $Y$ are define by (\ref{PolaEx}). It is easily seen that the
image under $p$ of the $x_{1}$-axis, denoted by $\Gamma $, is a closed and
singular trajectory for $(\Sigma )$. Indeed, its extremal lift is given by $%
\lambda (t)=(t\func{mod}2\pi ,0,0,0,0,1)$.

Define a rank $2$ distribution $H$ on $M$ by letting $H=Span\{X,Y\}$. If $x$
is a point in $M$ and $l$ is a positive integer, then we will write $%
H_{x}^{l}$ for the span of all vectors of the form 
\begin{equation*}
\lbrack X_{1},[X_{2},...,[X_{i-1},X_{i}]...]](x)\text{,}
\end{equation*}%
where $X_{1},...,X_{i}$ are smooth local sections of $H$ defined near $x$, $%
i\leq l$. Now it is not difficult to see that if $S=\left\{ x_{2}=0\right\} $%
, then $H$ is a contact distribution on $M\backslash S$, i.e. $%
H_{x}^{2}=T_{x}M$ whenever $x\in M\backslash S$. It can also be seen that $H$
has the following bracket properties on $S$: $H_{x}^{l}\subset H_{x}$, $%
1\leq l\leq k$, and $H_{x}^{k+1}=T_{x}M$ whenever $x\in S$. All this permits
us to conclude that, as it is explained in \cite{gr}, $(\Sigma )$ is an
affine control system induced by the generalized Martinet sub-Lorentzian
structure of Hamiltonian type of order $k$. Suppose that $k$ is odd. It
follows \cite{gr} that for every $x_{0}\in \Gamma $ there exists a
neighbourhood $U$ of $x_{0}$ and coordinates $\tilde{x}_{1},\tilde{x}_{2},%
\tilde{x}_{3}$ on $U$, $\tilde{x}_{1}(x_{0})=\tilde{x}_{2}(x_{0})=\tilde{x}%
_{3}(x_{0})=0$, such that $S\cap U=\left\{ \tilde{x}_{2}=0\right\} $, $%
\Gamma \cap U=\left\{ \tilde{x}_{2}=\tilde{x}_{3}=0\right\} $, and $\mathcal{%
A}^{+}(x_{0},U)=A_{1}\cup A_{2}$, where\smallskip $\newline
A_{1}=\left\{ x\in U:\text{ }\eta _{1}\left( \tilde{x}_{1}(x),\tilde{x}%
_{2}(x),\tilde{x}_{3}(x)\right) \leq 0\right\} \cap \left\{ \tilde{x}%
_{1}(x)\geq 0\text{, }\tilde{x}_{3}(x)\geq 0\right\} $,\newline
$A_{2}=\left\{ x\in U:\text{ }\eta _{2}\left( \tilde{x}_{1}(x),\tilde{x}%
_{2}(x),\tilde{x}_{3}(x)\right) \leq 0\right\} \cap \left\{ \tilde{x}%
_{1}(x)\geq 0\text{, }\tilde{x}_{3}(x)\leq 0\right\} $,\smallskip \newline
with \smallskip \newline
$\eta _{1}\left( \tilde{x}_{1},\tilde{x}_{2},\tilde{x}_{3}\right) =\tilde{x}%
_{3}+\frac{1}{2k}(\tilde{x}_{1}+\tilde{x}_{2})\left( \tilde{x}_{2}^{k}-\frac{%
1}{2^{k}}(\tilde{x}_{1}+\tilde{x}_{2})^{k}\right) +O(r^{k+2})$, \newline
$\eta _{2}\left( \tilde{x}_{1},\tilde{x}_{2},\tilde{x}_{3}\right) =-\tilde{x}%
_{3}-\frac{1}{2k}(\tilde{x}_{1}-\tilde{x}_{2})\left( \tilde{x}_{2}^{k}+\frac{%
1}{2^{k}}(\tilde{x}_{1}-\tilde{x}_{2})^{k}\right) +O(r^{k+2})$; \smallskip 
\newline
here $r=(\tilde{x}_{1}^{2}+\tilde{x}_{2}^{2}+\tilde{x}_{3}^{2})^{1/2}$. 
\newline
Since $\eta _{1}(\tilde{x}_{1},0,0)<0$ and $\eta _{2}(\tilde{x}_{1},0,0)<0$
(we choose $U$ to be sufficiently small), it is seen that $\Gamma
_{x_{0}}\cap U\subset int$ $\mathcal{A}^{+}(x_{0},U)$, and $\Gamma $ is a
regular closed orbit. At the same time one easily sees that $[\tilde{X},%
\tilde{Y}]=-kx_{2}^{k-1}\frac{\partial }{\partial x_{3}}$ which yields $%
ad^{l}\tilde{X}.\tilde{Y}=0$ for all $l\geq 2$, meaning that (\ref{Arwar})
does not hold at any point of $\Gamma $.

\begin{acknowledgement}
This work was partially supported by the Polish Ministry of Research and
Higher Education, grant NN201 607540.
\end{acknowledgement}

\textsl{Faculty of Mathematics and Natural Sciences, Cardinal Stefan Wyszy%
\'{n}ski University, ul. Dewajtis 5, 01-815 Waszawa, Poland.}

\textsl{Institute of Mathematics, Polish Academy of Sciences, ul. \'{S}%
niadeckich 8, 00-950 Warszawa, Poland.}

\end{document}